\documentclass[a4paper,twoside,12pt]{article}
\usepackage[margin=3cm]{geometry}

\usepackage[latin1]{inputenc}
\usepackage{amssymb}
\usepackage{latexsym}
\usepackage{amsfonts}
\usepackage{amsmath}
\usepackage{bm}
\usepackage{bbm}
\usepackage{mathrsfs}
\usepackage[T1]{fontenc}
\usepackage{ae}
\usepackage[english]{varioref}
\usepackage{graphicx}
\usepackage{rotating}
\usepackage{mathdots,mathtools,placeins,booktabs,array}

\newcommand{\N}{\mathbbm{N}}

\renewcommand{\pmod}[1]{\, (\operatorname{mod } #1)}
\renewcommand{\bar}{\overline}

\newcommand{\trace}{\operatorname{trace}}

\let\part\relax 
\DeclarePairedDelimiter\part{\langle}{\rangle}
\DeclarePairedDelimiter\res[]

\newcommand\mvdots{\mathmakebox[\widthof{${}={}$}]{\vdots}}

\usepackage{amsthm}

\theoremstyle{plain}
\newtheorem{thm}{Theorem} 
\newtheorem*{thm*}{Theorem} 
\newtheorem{prop}[thm]{Proposition}
\newtheorem{lemma}[thm]{Lemma}
\newtheorem{cor}[thm]{Corollary}

\theoremstyle{definition}
\newtheorem{defn}[thm]{Definition}
\newtheorem{example}[thm]{Example}

\theoremstyle{remark}

\newtheorem*{remark*}{Remark}
\newtheorem*{claim*}{Claim}

\newcommand{\refthm}[1]{Theorem~\ref{#1}}

\newcommand{\reflemma}[1]{Lemma~\ref{#1}}
\newcommand{\refcor}[1]{Corrolary~\ref{#1}}

\title{On lower bounded orbits of the times-q map}
\author{Jonas Lindstr\o{}m Jensen\footnote{Email: jonas@imf.au.dk}}
\date{}

\begin{document}
\maketitle
\begin{abstract}
In this paper we consider the times-$q$ map on the unit interval as a subshift of finite type by identifying each number with its base $q$ expansion, and we study certain non-dense orbits of this system where no element of the orbit is smaller than some fixed parameter $c$.

The Hausdorff dimension of these orbits can be calculated using the spectral radius of the transition matrix of the corresponding subshift, and using simple methods based on Euclidean division in the integers, we completely characterize the characteristic polynomials of these matrices as well as give the value of the spectral radius for certain values of $c$. It is known through work of Urbanski and Nilsson that the Hausdorff dimension of the orbits mentioned above as a map of $c$ is continuous and constant almost everywhere, and as a new result we give some asymptotic results on how this map behaves as $q \to \infty$.
\end{abstract}

\section{Introduction}
In this paper we study the set
\[ F_c^q = \{x \in [0,1) \mid \{q^n x\} \geq c \text{ for all } n \geq 0\} \]
where $q \geq 2$ is an integer and $\{\cdot\}$ denotes the fractional part. This set is related to badly approximable numbers in Diophantine approximation, and has been studied by Nilsson \cite{lit:nilsson}, who studied the Hausdorff dimension of the set as a map of $c$, and in more generality by Urbanski \cite{lit:urbanski} who considered the orbit of an expanding map on the circle.

As Nilsson did we will consider $F_c^q$ as a subshift of finite type which enables us to see it as a problem in dynamical systems. When studied as a subshift of finite type we can find the dimension of $F_c^q$ using the spectral radius of the corresponding transition matrix, and this motivates the theorem of this paper which characterizes the characteristic polynomial of this matrix.

The author would like to thank his PhD supervisor Simon Kristensen and he would also like to than Johan Nilsson for
reading and commenting on an early version of this paper. 

\section{Basic definitions}
We begin with a definition of part and residue which comes from elementary integer division with residue. We let $q \geq 2$ be an integer throughout the paper and start with a well known result.
\begin{prop} For integers $n \in \N$ and $m \geq 0$ there are unique integers $\part{n,m} \in \N$ \emph{(part)} and $\res{n,m}$ (residue), with $0 \leq \res{n,m} < q^m$ such that
\[ n = q^m \part{n,m} + \res{n,m}. \]
\end{prop}
We note that if we write $n = n_k \cdots n_1$ in base $q$ it is easy to find the part and the residue, since $\res{n,m} = n_m \cdots n_1$ and $\part{n,m} = n_k \cdots n_{m+1}$.

The matrix we will consider in this paper is defined as follows.
\begin{defn} For $m \geq 1$ we define a 0-1 matrix $A_m$ of size $q^m \times q^m$ by
\[ (A_m)_{ij} = 1 \iff \res{i-1,m-1} = \part{j-1,1}. \]
We let $A_m(P)$ with $P \subseteq \{1,2,\ldots,q^m\}$ be the $\# P \times \# P$ matrix made from picking only the rows and columns from $A_m$ corresponding to the elements in $P$ and for $0 \leq k \leq m$ we let $A_m(k)$ be the $(m-k) \times (m-k)$ matrix where we have removed the first $k$ rows and columns from $A_m$.
\end{defn}
We will often omit the dependency on $m$ when it is not confusing. Considering $i$ and $j$ in base $q$ we see that $(A_m)_{ij} = 1$ if and only if the first $m-1$ digits of $j-1$ are equal to the last $m-1$ digits of $i-1$. So when $c = \frac{i}{q^m}$ we see that the base $q^m$ expansions of the numbers in $F_c^q$ can be seen as a subshift of finite type with transition matrix $A_m(i)^m$. The metric of the subshift and the unit interval are equivalent so the dimensional properties are the same. In particular, finding the Hausdorff dimension of $F_c^q$ now boils down to finding the spectral radius $\rho(A_m(k))$, since
\begin{equation}\label{eq:hausd} \dim_H F(c) = \frac{\rho (A_m(i)^m)}{\log q^m} = \frac{\rho (A_m(i))}{\log q}. \end{equation}
For a proof of the first equality see \cite{lit:pesin}. This is why we were interested in finding the characteristic polynomial of $A_m(i)$. The main theorem of this paper is a complete characterization of these polynomials. In order to state this theorem we need the following definition.
\begin{defn} For integers $n,m \geq 1$ with $0 \leq n < q^m$ we define
\[ l_m(n) = \min\{j \in \N \mid 1 \leq j \leq m \text{ and } \part{n,j} \geq \res{n, m-j} \}. \]
Using this definition we let
\[ \bar{n}_m = n - \res{n,m-l_m(n)} = q^{m-l_m(n)} \part{n,m-l_m(n)} \]
be the \emph{minimal prefix} of $n$. 
\end{defn}
This is well defined since $\res{n,0} = \part{n,m} = 0$ for any $n$ with $0 \leq n < q^m$. The notion of minimal prefix is taken from Nilsson \cite{lit:nilsson}, but is here defined somewhat differently since we only consider finite sequences.

Let us consider some examples.
\begin{example} Let $q=3, m=3$. Then
\[ \part{11,1} = 3 \geq 2 = \res{11,2} \]
so $l_3(11)=1$ and
\[ \bar{11}_3 = 11 - \res{11,2} = 9. \]
If we let $n=7$ we have
\[ \part{7,1} = 2 < 7 = \res{7,2} \]
and
\[ \part{7,2} = 0 < 1 = \res{7,1} \]
but
\[ \part{7,3} = 0 = \res{7,0} \]
so $l_3(7)=3$ and $\bar{7}_3 = 7$.
\end{example}

We are now ready to state the main theorem.
\begin{thm}\label{thm:char} Let $i$ be an integer such that $0 < i < q^m$ and let $f_i^m(x)$ be the characteristic polynomial of $A_m(i)$. Then 
\[ f_i^m(x) = g_i^m(x) x^{q^m-m-i} \]
where
\[ g_i^m(x) = x^m - a_1 x^{m-1} - \cdots - a_m \]
and $a_1 a_2 \ldots a_m$ is the base $q$ expansion of $q^m-\bar{i}_m$.
\end{thm}
Notice that this implies the equality
\[ g_i^m(q) = \bar{i}_m, \]
and that
\[ \dim_H F_{i/q^m} = \frac{\log \rho}{\log q} \]
where $\rho$ is the Perron root of $g_i^m$.

\section{Proof outline}
First recall that we can find the characteristic polynomial $f_i^m(x) =
x^{q^m-i} - a_1 x^{q^m-i-1} \allowbreak - \cdots - a_{q^m-i}$ of $A_m(i)$ as
\begin{equation}\label{eq:minors} 
	a_j = (-1)^j \sum_{\# P = j} \det A_m(P)
\end{equation}
where we also require that $P \subseteq \{i+1,i+2,\ldots,q^m\}$, or as
\begin{equation}\label{eq:recur} 
  a_j = \frac{1}{j} \big( \trace A_m(i)^j + a_1 \trace A_m(i)^{j-1} +
  \cdots + a_{j-1} \trace A_m(i) \big). 
\end{equation}
The first formula is sometimes used as the definition of the characteristic polynomial, and for a proof of the latter see \cite{lit:faddeev}. We now try to outline the proof that essentially is the construction of an algorithm that calculates both the characteristic polynomial of $A_m(i)$ and $\bar i_m$.
\begin{itemize}
\item We prove that all the submatrices $A(P)$ that give non-zero principal minors are permutations, so when removing rows and columns from the first to the last, we only change the characteristic polynomial when removing rows and columns corresponding to the smallest element of a cycle.
\item If $l_m(i) = m$ then $i$ is the smallest element of an $m$-cycle and this is the only permutation of size $\leq m$ that has $i$ as an element. So removing $i$ decreases the $m$'th coefficient of the characteristic polynomial by $1$ and leaves all the preceding coefficients unchanged. On the other hand, if $l_m(i) = n < m$, then the nontrivial part of the characteristic polynomial, $g_i^m(x)$, can be found as $x^{m-n} g_{\part{i,m-n}}^n(x)$ since we have \eqref{eq:recur} and can prove that
\[ \trace A_m(i)^k = \trace A_n(\part{i,m-n})^k \]
for all $k \leq m$.
\item If $l_m(i) = m$, then $\overline{i}_m = \overline{i+1}_m - 1$, and if $l_m(i) = n < m$ then $\overline{i}_m = q^{m-n} \overline{\part{i,m-n}}_n$, so we see that $\bar i$ and the characteristic polynomials follow the same pattern.
\item Since the theorem is true for $m=1$, we can now use induction if $l_m(i) < m$. If not, we increase $i$ until we have $l_m(i) < m$, which happens since $l_m(q^m-1) = 1$.
\item The $m+1$'st, $m+2$'nd, \dots, $q^m$'th coefficient of $f_i^m(x)$
  are all zero, because we have found the first $M$ coefficients of
  the characteristic polynomial for any $M$, so if we pick $M > m$ and $K$ such that
  $l_M(K) = m$ and $\part{K,M-m} = i$, then we see that $g_K^M(x)$ has
  its $m+1$'th, $m+2$'th, \dots, $M$'th coefficients equal to zero, which will then also be true for $g_i^m(x)$. This finishes the proof of the theorem.
\end{itemize} 

\begin{sidewaystable}
\caption{Calculation of the characteristic polynomial of $A_3(i)$
  when $q=3$. We let $g_i^3(x) = x^3-a_1x^2-a_2x-a_3$. We also give the
  minimal prefix and the length of the minimal prefix. The numbers in
  bold indicates that we consider a minimal number with non-maximal
  prefix length $l(i) < 3$.} 
\addtolength\arraycolsep{1.3pt}
\[ \begin{array}{c !\quad  ccccccccccccccccccccccccccc}
\toprule
i & 0 & 1 & 2 & 3 & 4 & 5 & 6 & 7 & 8 & 9 & 10 & 11 & 12 & 13 & 14 & 15 & 16 & 17 & 18 & 19 & 20 & 21 & 22 & 23 & 24 & 25 & 26 \\ 
\midrule
\bar{i} & 0 & 1 & 2 & 3 & 4 & 5 & 6 & 7 & 8 & 9 & 9 & 9 & 9 & 9 & 14 & 15 & 15 & 17 & 18 & 18 & 18 & 18 & 18 & 18 & 18 & 18 & 18 \\
\midrule
a_1 & 3 & 2 & 2 & 2 & 2 & 2 & 2 & 2 & 2 & 2 & 2 & 2 & 2 & 2 & 1 & 1 & 1 & 1 & 1 & 1 & 1 & 1 & 1 & 1 & 1 & 1 & 0 \\
a_2 & 0 & 2 & 2 & 2 & 1 & 1 & 1 & 0 & 0 & 0 & 0 & 0 & 0 & 0 & 1 & 1 & 1 & 0 & 0 & 0 & 0 & 0 & 0 & 0 & 0 & 0 & 0 \\
a_3 & 0 & 2 & 1 & 0 & 2 & 1 & 0 & 2 & 1 & 0 & 0 & 0 & 0 & 0 & 1 & 0 & 0 & 1 & 0 & 0 & 0 & 0 & 0 & 0 & 0 & 0 & 0 \\
\midrule
\addlinespace[1.5ex]
\midrule
l(i)& 1 & 3 & 3 & 2 & 3 & 3 & 2 & 3 & 3 & 1 & 1 & 1 & 1 & 1 & 3 & 2 & 2 & 3 & 1 & 1 & 1 & 1 & 1 & 1 & 1 & 1 & 1 \\
\midrule
A_{i+1,i+1}
    & \mathbf{1} & 0 & 0 & 0 & 0 & 0 & 0 & 0 & 0 & 0 & 0 & 0 & 0 & \mathbf{1} & 0 & 0 & 0 & 0 & 0 & 0 & 0 & 0 & 0 & 0 & 0 & 0 & 1 \\
A_{i+1,i+1}^2
    & 1 & 0 & 0 & \mathbf{1} & 0 & 0 & \mathbf{1} & 0 & 0 & 0 & 1 & 0 & 0 & 1 & 0 & 0 & \mathbf{1} & 0 & 0 & 0 & 1 & 0 & 0 & 1 & 0 & 0 & \mathbf{1} \\
\bottomrule
\end{array} \]
\end{sidewaystable}

\section{Part and residue}
The results in this sections explain some properties of the part and residue functions and gives a characterization of the powers of $A$. We will use these results throughout the paper, often without specifically stating so. The proofs in this section are rather straightforward and may be skipped on a first read.
\begin{prop} 
\begin{enumerate}
\item For $j,k,n \geq 0$ we have $\res{\res{n,j}, k} = \res{n, \min \{j,k\}}$ and
\[ 
\part[\big]{\part{n,k}, j} = \part{n,k+j}. 
\]
\item For $j>k$ we have
\[ \part{\res{n,j},k} = \res{\part{n,k},j-k}. \]
\end{enumerate}
\end{prop}
\begin{proof}
  Let us first prove the two equalities in 1. Since $\res{n,k}$ is the
  same as $n \pmod q^k$ we have the first equality. Now assume that
  $j+k \leq m$. Now $\part{n,k} = q^j \part[\big]{\part{n,k},j} +
  \res{\part{n,k},j}$, so
\[ 
n = q^k \part{n,k} + \res{n,k} = q^{k+j} \part[\big]{\part{n,k},j} +
q^k\res{\part{n,k},j} + \res{n,k},
\] 
but since $\res{\part{n,k},j} < q^j$ and $\res{n,k} < q^k$ we have
\[
q^k\res{\part{n,k},j} + \res{n,k} \leq q^k(q^j-1) + q^k-1 = q^{k+j} -
1 < q^{k+j},
\]
and by the uniqueness of the residue and parts we see that
$\part{\part{n,k},j} = \part{n,k+j}$. Now consider 2., so let $j > k$.
>From 1. we have
\[ 
\part{n,k} = q^{j-k} \part[\big]{\part{n,k},j-k} +
\res{\part{n,k},j-k} = q^{j-k} \part{n,j} + \res{\part{n,k},j-k} 
\]
and
\[ 
\res{n,j} = q^k \part{\res{n,j},k} + \res[\big]{\res{n,j},k} = q^k
\part{\res{n,j},k} + \res{n,k}. 
\]
So
\begin{align*} n 
&= q^k \part{n,k} + \res{n,k} \\
&= q^j \part{n,j} + q^k \res{\part{n,k},j-k} - q^k \part{\res{n,j},k} + \res{n,j} \\
&= q^j \part{n,j} + \res{n,j} + q^k (\res{\part{n,k},j-k} - \part{\res{n,j},k})
\end{align*}
and since $n = q^j \part{n,j} + \res{n,j}$ this implies that
\[ \res{\part{n,k},j-k} = \part{\res{n,j},k}. \]
\end{proof}

\begin{lemma}\label{lem:powersofA} Let $k$ be an integer with $1 \leq k \leq m$. Then $A^k_{ij} = 1$ if and only if
\[ \res{i-1,m-k} = \part{j-1,k}. \]
\end{lemma}
\begin{proof} We will prove this by induction. For $k=1$ it is the definition of $A$, so assume that $1 < k \leq m$. We assume that the lemma is true for all smaller $k$. If $A^k_{ij} = 1$ there must exist some $n$ with $0 \leq n < q^m$ and $A_{nj} = 1$ and $A^{k-1}_{in} = 1$. Using the induction hypothesis we get
\begin{equation}\label{eq:eqs1} \res{i-1,m-k+1} = \part{n-1,k-1} \quad \text{and} \quad \res{n-1,m-1} = \part{j-1,1} \end{equation}
for this $n$. Now by part 2. of the above proposition we have
\[ \res{\part{n-1,k-1},m-k} = \part{\res{n-1,m-1},k-1}, \]
and using \eqref{eq:eqs1} we get
\[ \res[\big]{\res{i-1,m-k+1}, m-k} = \part[\big]{\part{j-1,1},k-1}, \]
and using part 1. of the proposition we get
\[ \res{i-1,m-k} = \part{j-1,k} \]
as desired.

Now assume that $\res{i-1,m-k} = \part{j-1,k}$. Let
\[ n-1 = q^{k-1} \res{i-1,m-k+1} + \res{\part{j-1,1},k-1}. \]
This is a positive integer smaller than $q^m$. By the uniqueness of the residue and parts we see that
\begin{equation}\label{eq:ni} \res{i-1,m-k+1} = \part{n-1,k-1} \end{equation}
and 
\begin{equation}\label{eq:jn} 
\res{\part{j-1,1},k-1} = \part{n-1,k-1}. 
\end{equation}
>From \eqref{eq:ni} and the induction hypothesis we see that
$A^{k-1}_{in} = 1$. We now want to prove that $A_{nj} = 1$. Recall
that we assume $\res{i-1,m-k} = \part{j-1,k}$, so
\begin{align*} \part{\res{n-1,m-1},k-1} 
&= \res{\part{n-1,k-1},m-k} \\
&= \res[\big]{\res{i-1,m-k+1},m-k} \\
&= \res{i-1,m-k} \\
&= \part{j-1,k}.
\end{align*}
Using this and \eqref{eq:jn} we see that
\begin{align*} \res{n-1,m-1} 
&= q^{k-1} \part{\res{n-1,m-1},k-1} + \res[\big]{\res{n-1,m-1},k-1} \\
&= q^{k-1} \part{j-1,k} + \res{n-1,k-1} \\
&= q^{k-1} \part{j-1,k} + \res{\part{j-1,1},k-1} \\
&= q^{k-1} \part[\big]{\part{j-1,1},k-1} + \res{\part{j-1,1},k-1} \\
&= \part{j-1,1}.
\end{align*}
This proves that $A^{k-1}_{in} = 1$ and $A_{nj}=1$ which implies that $A^k_{ij} > 0$. Now assume that there is another $n'$ such that $A^{k-1}_{in'} = 1$ and $A_{n'j}=1$. Then
\[ \res{i-1,m-k+1} = \part{n'-1,k-1} \]
and 
\[ \res{\part{j-1,1},k-1} = \part{n'-1,k-1} \]
so
\begin{align*} n'-1
  &= q^{k-1} \part{n'-1,k-1} + \res{n'-1,k-1} \\
  &= q^{k-1} \res{i-1,m-k+1} + \res[\big]{\res{n'-1,m-1},k-1} \\
  &= q^{k-1} \res{i-1,m-k+1} + \res{ \part{j-1,1},k-1} \\
  &= n-1,
\end{align*}
which proves that there can be only one such $n$, so $A^k_{ij} = 1$.
\end{proof}

\begin{lemma} \label{lem:ineqs} If $a,b,k$ is such that $\res{a,k} < \res{b,k}$ and $\part{a,k} = \part{b,k}$, then 
\[ \res{a,k+j} < \res{b,k+j} \]
for all integers $j$ with $0 \leq j \leq m-k$.
\end{lemma}
\begin{proof} If $\part{a,k} = \part{b,k}$ then
\[ \part[\big]{\part{a,k},j} = \part[\big]{\part{b,k},j}, \]
and hence
\[ \part{a,k+j} = \part{b,k+j}. \]
Since $a < b$ we thus have
\[ \res{a,k+j} < \res{b,k+j} \]
as desired.
\end{proof}

\section{Minimality}
We now prove the following rather simple lemma which states that the only non-zero principal minors can be found as submatrices of $A$ which are permutations.
\begin{lemma}\label{lem:mustbeperm} If $\det A(P) \neq 0$ then the corresponding matrix is a permutation matrix. \end{lemma}
\begin{proof} Assume that we choose $P$ such that one of the rows of $A(P)$ has two ones. In other words there are $i, j_1, j_2 \in P$ such that
\[ A_{ij_1} = A_{ij_2} = 1. \]
Using the definition of $A$ this implies that
\[ \part{j_1-1,1} = \res{i-1,m-1} = \part{j_2-1,1}. \]
Now let $k \in P$ be arbitrary. Then $A_{kj_1}=1$ if and only if $\res{k-1,m-1} = \part{j_1-1,1}$, which is true if and only if
\[ \res{k-1,m-1} = \part{j_2-1,1}, \]
so $A_{kj_1} = A_{kj_2}$ for all $k \in P$, so the $j_1$'th and $j_2$'th column are equal and so $\det A(P) = 0$. The proof is similar when we assume that there are two ones in one column.
\end{proof}

Recall that if $A(P)$ is a permutation, then $P = P_1 \cup \cdots \cup P_n$ where $\cap_i P_i = \emptyset$ and $A(P_i)$'s are all cycles. This motivates the following two theorems, where we characterize the subsets $P$ where $A(P)$ is a cycle. We are interested in the smallest elements of cycles, since the whole cycle are removed when we remove this element, which we will prove is exactly the numbers that are minimal.
\begin{defn} We say that an integer $n$ with $0 \leq n \leq q^m$ is $m$-minimal if
\[ A^{l(n)}_{n+1,n+1} = 1, \]
or equivalently using \reflemma{lem:powersofA} if
\[ \res{n,m-l(n)} = \part{n,l(n)}. \]
\end{defn}

\begin{thm} Let $P \subset \{1,2,\ldots,q^m\}$ be such that $A(P)$ is a $k$-cycle for some $1 \leq k \leq m$. Then $\min P-1$ is minimal with $l_m(\min P-1)=k$. \end{thm}
\begin{proof} Let $P = \{i_1,i_2,\ldots,i_k\}$ be a $k$-cycle with $A_{i_1 i_{j+1}}^{j} = 1$ for $1 \leq j < k$ and $A_{i_1 i_1}^k = 1$. Without loss of generality we can assume that $\min P = i_1$. Using \reflemma{lem:powersofA} we get that
\[ \res{i_1-1,m-j} = \part{i_{j+1}-1,j}, \]
for $1 \leq j < k$ and
\[ \res{i_1-1,m-k} = \part{i_1-1,k} \]
so we need to prove that $\part{i_{j+1}-1, j} > \part{i_1-1, j}$ for $j=1,2,k-1$. We have the non-strict inequality since $i_1 < i_j$. So assume for contradiction that
\[ \part{i_1-1,j} = \part{i_{j+1}-1,j}. \]
Now since $i_1 < i_{j+1}$ we have
\[ \res{i_1-1,j} < \res{i_{j+1}-1,j}, \]
and due to \reflemma{lem:ineqs} we have
\begin{equation}\label{eq:ineq1} \res{i_1-1,m-k+j} < \res{i_{j+1}-1,m-k+j} \end{equation}
since $k \leq m$. Since $A^{k-j}_{i_{j+1} i_1} = 1$ we have $\res{i_{j+1}-1,m-k+j} = \part{i_1-1,k-j}$. Using \eqref{eq:ineq1} we get
\[ \res{i_1-1, m-k+j} < \part{i_1-1,k-j}. \]
Now consider $i_{k-j+1}$. Since $j < k$ we have $A^{k-j}_{i_1 i_{k-j+1}} = 1$ so 
\[ \res{i_1-1,m-k+j} = \part{i_{k-j+1}-1, k-j}, \]
and hence
\[ \part{i_{k-j+1}-1, k-j} < \part{i_1-1,k-j}. \]
This implies that $i_{k-j+1} < i_1$ which is a contradiction against $i_1$ being the least element in $P$.
\end{proof}

\begin{thm}\label{thm:uniquecycle} Assume that $i-1$ is minimal. Then
  there is a unique $P \subseteq \{1,2,\ldots,q^m\}$ such that $\min P
  = i$ and $A(P)$ is a $l(i-1)$-cycle. 
\end{thm}
\begin{proof} We let $P = \{i, i_2, i_3, \ldots, i_k\}$ where
\begin{align*}
i_2-1 &= q \res{i-1,m-1} + \part{i-1,m-1} \\
i_3-1 &= q^2 \res{i-1,m-2} + \part{i-1,m-2} \\
&
\mvdots \\
i_k-1 &= q^{k-1} \res{i-1,m-k+1} + \part{i-1,m-k+1}.
\end{align*}
We now need to prove that $A_{ii_n}^{n-1} = 1$ and that $i < i_n$ for all $n=2,3,\ldots,k$. Using the uniqueness of the part and residue we see that
\[ \part{i_n-1,n-1} = \res{i-1,m-n+1} \]
and
\[ \res{i_n-1,n-1} = \part{i-1,m-n+1} \]
for $n=2,3,\ldots,k$. The first of these equations implies that $A_{ii_n}^{n-1} = 1$.

Since $l_m(i-1) = k$ we know that
\[ \part{i-1,n} < \res{i-1,m-n} \]
for $n=1,2,\ldots,k-1$. This implies that
\[ i_{n+1}-1 = q^n \res{i-1,m-n} + \part{i-1,m-n} > q^n \part{i-1,n} + \res{i-1,n} = i-1 \]
since both $\part{i-1,m-n}$ and $\res{i-1,n}$ are smaller than $q^n$.

We now need to prove that this $P$ is unique. Assume that we have $P' = \{i,i_2', \ldots, i_k'\}$, where we order the elements such that $A_{ii_n'}^{n-1} = 1$. This implies that
\[ \res{i-1,m-n+1} = \part{i_n'-1,n-1} \]
for all $n=2,3,\ldots,k$. Since $A(P)$ is a $k$-cycle, we furthermore know that $A_{i_n'i}^{k-n+1} = 1$, so
\[ \res{i_n'-1,m-k+n-1} = \part{i-1, k-n+1}. \]
Now we want to prove that $i_n' = i_n$, so let $2 \leq n \leq k$ be given. We have
\[ i_n'-1 = q^{n-1} \part{i_n'-1,n-1} + \res{i_n'-1, n-1} \]
and $\part{i_n'-1,n-1} = \res{i-1,m-n+1}$, so we just need to prove that
\[ \res{i_n'-1,n-1} = \part{i-1,m-n+1}. \]
We have
\begin{align*} \res{i_n'-1,n-1}
&= \res[\big]{\res{i_n'-1,m-k+n-1}, n-1} \\
&= \res{\part{i-1,k-n+1}, n-1} \\
&= \res[\big]{\res{i_n-1,m-k+n-1}, n-1} \\
&= \res{i_n-1, n-1} \\
&= \part{i-1,m-n+1}
\end{align*}
so $i_n = i_n'$ for all $n$, and so $P = P'$.
\end{proof}

\begin{cor} If $l_m(i-1) = m$ then there is exactly one $P \subseteq \{1,2,\ldots,q^m\}$ such that $\min P = i$ and $A(P)$ is a $m$-cycle. \end{cor}
\begin{proof} This follows from the fact that $A^m_{ij} = 1$ for all $i,j$. In particular we have $A^m_{ii} = 1$ for all $i$.\end{proof}

Now compare this corollary with the following lemma.
\begin{lemma}\label{lem:minimaldecrease} If $l_m(i-1)=m$, then $\overline{i}_m = \overline{i-1}_m+1$. \end{lemma}
\begin{proof}
It is enough to prove that $\bar i = i$, since we certainly have $\overline{i-1} = i-1$. Using the definition we see that this is equivalent to $\res{i,m-l(i)} = 0$. If $l(i) = m$ we are done, so assume that $l(i) < m$. Now either $\res{i,m-l(i)} = 0$, in which case we are done, or $\res{i,m-l(i)} = \res{i-1,m-l(i)}+1$. Now since $l(i-1) = m$ we have
\[ \res{i-1,m-l(i)} < \part{i-1,l(i)}, \]
since $l(i) < m = l(i-1)$, but
\[ \res{i-1,m-l(i)} = \res{i,m-l(i)} - 1 \leq \part{i,l(i)} - 1 \leq \part{i-1,m-l(i)}, \]
which is a contradiction.
\end{proof}
Recalling the idea of the proof we here see that if $l_m(i-1) = m$ and
we remove the $i$'th row and column of $A_m$, then we remove exactly
one permutation of size $\leq m$, namely an $m$-cycle, which increases
the $m$'th coefficient of the characteristic polynomial by one, and we
also see that it increases the $m$'th digit of the base $q$ expansion
of $\bar i$ by one.

\section{Induction mapping}
In the following chapter we will no longer suppress the dependency on $m$, since we are interested in mapping permutations between matrices of different sizes while preserving cycles. We will illustrate the idea with an example. If $q=3$, and we write all numbers in base $3$ we see that
\begin{equation}\label{eq:updownex} 012,120,201 \end{equation}
is a $3$-cycle in $A_3(012)$. We now map this up to
\[ 0120,1201,2012 \]
which is a $3$-cycle in $A_4(0120)$. On the other hand we could also map \eqref{eq:updownex} down to
\[ 01,12,20 \]
which is a $3$-permutation in $A_2(01)$. In this section we will formally define these maps, and also prove that they map cycles to cycles. We begin with the `down' map which is defined in the following way.
\begin{defn} For an integer $i$ with $0 \leq i < q^{m+1}$ we define
\[ D_m(i) = \part{i,1}. \]
If $M > m$ and $0 \leq i \leq q^M$ we let
\[ D_{m,M}(i) = D_m \circ \cdots \circ D_{M-1}(i) = \part{i,M-m}. \]
\end{defn}

We now prove the following lemma.
\begin{lemma} If we for integers $i,M$ have $l_M(i) = m < M$, then
\[ l_m(D_{m,M}(i)) = m. \]
\end{lemma}
\begin{proof} We have $\res{i, M-m} \geq \part{i,m}$ and $\res{i,M-j} < \part{i,j}$ for all $1 \leq j < m$, and we need to prove that $\res{i,m-j} < \part{i,j}$ for all $1 \leq j \leq m$. But this is clearly the case since $m < M$, so
\[ \res{i,m-j} < \res{i,M-j} < \part{i,j} \]
for all $1 \leq j < m$.
\end{proof}

\begin{cor}\label{cor:recmin} Let $i$ be an integer with $0 \leq i < q^M$. If $l_M(i) = m < M$, then
\[ \bar i_M = q^{M-m} \overline{D_{m,M}(i)}_m. \]
\end{cor}
\begin{proof}
This follows from the definition of the minimal prefix.
\end{proof}

We saw earlier that the characteristic polynomial of a matrix can be found by considering the trace of the powers of the matrix. So if we can map permutations bijectively between two transition matrices we must have the same characteristic polynomials. As before we only need to consider cycles as all permutations are products of cycles.
\begin{defn} An ordered $k$-tuple of distinct elements, $(i_1,\ldots,i_k)$ with $0 \leq i_j \leq q^m$ for all $j=1,2,\ldots,k$ is a $k$-cycle in $A_m(c)$ if $A_m(c)_{i_j, i_{j+1}} = 1$ for all $j=1,2,\ldots,k-1$, and $A_m(c)_{i_k, i_1} = 1$. In other words, if we have
\[ \res{i_j,m-1} = \part{i_j,1} \]
for $j=1,2,\ldots,k-1$ and $\res{i_k,m-1} = \part{i_1,1}$ and $i_j \geq c$ for all $j=1,2,\ldots,k$.
\end{defn}

We have a `down' map, mapping from large matrices to smaller and we now define an `up' map, mapping from smaller to larger.
\begin{defn} Let $P = (i_1,\ldots,i_k)$ be a $k$-cycle in $A_m(c)$. Then we let
\[ U_m(P) = (qi_1 + \res{i_2,1}, \cdots, qi_k + \res{i_1,1}), \]
and for $M > m$ we let $U_{m,M} = U_{M-1} \circ U_{M-2} \circ \cdots \circ U_m$.
\end{defn}

\begin{lemma} Let $m = l_M(c)$ and let $P = (i_1,i_2,\ldots,i_k)$ be a $k$-cycle in $A_M(c)$. Then
\[ D_{m,M}(P) = (D_{m,M}(i_1), \cdots, D_{m,M}(i_k)) \]
is a $k$-cycle in $A_m(D_{m,M}(c))$. Furthermore, if $Q = (j_1,\ldots,j_k)$ is a $k$-cycle in $A_m(D_{m,M}(c))$, then $U_{m,M}(Q)$ is a $k$-cycle in $A_M(c)$.
\end{lemma}
\begin{proof} To prove that $D_{m,M}(P)$ is a $k$-cycle in $A_m(D_{m,M}(c))$ can be done by straightforward calculations. We also get that $U_{m,M}(Q)$ is a $k$-cycle in $A_M(q^{M-m} \part{c, M-m})$ rather straightforward. The problem is to prove that it actually is a $k$-cycle in $A_M(c)$, or in other words that there are no $k$-cycles with their smallest element in the interval between $q^{M-m} \part{c,M-m}$ and $c$. Recalling the definition of $\bar c_M$ and that the least element of a cycle always is minimal we thus need to prove that if we have $\bar c_M \leq n < c$, then $n$ cannot be minimal.

We get that $\bar n_M = \bar c_M$ and $l_M(n) = l_M(c)$ so
\[ \res{c,M-m} - \res{n,M-m} = c-n \]
so if we assume that $n$ is minimal we get
\[ \part{c,m}  \geq \res{c,M-m} = \res{n,M-m} + c-n = \part{n,m} + c-n \]
which is a contradiction. This finishes the proof of the theorem.
\end{proof}

These two lemmas now lead to the following theorem regarding the invariance of the traces.
\begin{thm}\label{thm:powertrace} Let $m,k \leq M$. Then
\[ \trace A_m(c)^k = \trace A_M(q^{M-m} c)^k. \]
More generally we have
\[ \trace A_m(\part{c,M-m})^k = \trace A_M(c)^k \]
whenever $l_M(c) \geq m$.
\end{thm}
\begin{proof} Each $k$-cycle contributes to the trace, and since the maps used in the lemmas map all $k$-cycles injectively, we get the theorem.
\end{proof}

Newton's formula for the characteristic polynomial gives us, that if 
\[ f_i^m(x) = x^n - a_1 x^{n-1} - \cdots - a_n = \det(xI - A_m(i)) \]
is the characteristic polynomial of $A_m(k)$ where $n = q^m - i$, then
\[ a_j = \frac{1}{j}
\big( \trace A_m(i)^j - a_1 \trace A_m(i)^{j-1} - \cdots - 
a_{j-1} \trace A_m(i) \big) \]
so the above theorem gives us that
\[ f_i^M(x) = x^{M-m} f_{q^{M-m}i}^m(x). \]
Combining this with the simple lemma below gives us the proof of the main theorem.
\begin{lemma} Let $n$ be an integer with $0 \leq n < q^m$. Then
\[ q \bar{n}_m = \bar{qn}_{m+1}. \]
\end{lemma}
\begin{proof}
We see that
\[ q \bar{n}_m = q( n - \res{n, m-l_m(n)}) = qn - \res{qn, m+1-l_m(n)}, \]
so we just need to prove that $l_{m+1}(qn) = l_m(n)$. Assume that $j = l_m(n)$. Then
\[ \part{qn, j} \geq q 
\part[\big]{\part{qn,j},1} = q \part{qn,j+1} = q 
\part{n,j} \geq q \res{n,m-j} = \res{qn,m}. \] Now assume that
$\part{qn,j} \geq \res{qn,m+1-j}$ for some $j > l_m(n)$. Then
\[ q\res{n,j} = \res{qn,j} \leq \part{qn,m+1-j} \]
so
\[ \res{n,j} \leq \part{\part{qn,m+1-j},1} = \part{n,m-j} \]
which is a contradiction.
\end{proof}

We are now ready to prove the main theorem.
\begin{proof}[Proof of \refthm{thm:char}] We prove this theorem using induction. If $m=1$ it is certainly true since $\bar i_1 = i$ for all $i$ with $0 \leq i < q$ and $A_1$ is the all one matrix of size $q \times q$.

We see that when choosing $m$ and $i > 0$ we have two possibilities: either we have $l(i-1) = m$ or $l(i-1) < m$. In the first case removing the $i$'th column and row only removes one non-zero minor, namely the unique $m$-cycle with $i$ as its minimal element given in \refthm{thm:uniquecycle}. In this case we also have that the last digit of $\overline{i-1}_m$ is $\res{i-1,1}$ which must be non-zero, so here we just decrease $a_m$ with $1$, so the first $m$ coefficients of the characteristic polynomial changes in the right way due to \reflemma{lem:minimaldecrease}. 

If we have $l(i-1) = n < m$ we see that we can find the characteristic polynomial of the smaller matrix of size $q^n$ instead and multiply it by $x^{m-n}$. As we see in \refcor{cor:recmin} this is also the case for $\bar{k}$. So by induction we are done.

Now we need to prove that the remaining coefficients are all zero. To
prove this we once again use \reflemma{thm:powertrace} to see that the
$M$'th coefficient of $f_k^m$ must be equal to the $M$'th coefficient
of $f_{q^{M-m} c}^M$ for any $M > m$. And here we see that the
$m+1$'th, $m+2$'th, \dots, and $M$'th coefficient all are zero,
since the $M$'th digit of the base $q$ expansion of
\[ q^M - \bar{q^{M-m} c}_M = q^{M-m}(q^m - \bar{c}_m) \]
is zero. This finishes the proof of the theorem.
\end{proof}

\section{Constant dimension}
Now define $\phi: c \mapsto \dim_H F(c)$. Recall from \eqref{eq:hausd} that when $c$ has finite base $q$ expansion we can calculate $\phi(c)$. Nilsson \cite{lit:nilsson} proved that this function is continuous and constant almost everywhere. Using the theorem we see that if we have $0 \leq i < j < q^m$ such that $\bar i_m = \bar j_m$ then
\[ \phi\left( \frac{i}{q^m} \right) = \phi\left( \frac{j}{q^m} \right) \]
and since $\phi$ is a decreasing function it must be constant on the interval
\[ \left[\frac{i}{q^m}, \frac{j}{q^m}\right]. \]
Now let $0 \leq i < q$ be given and let 
\[ j(m) = \sum_{n=1}^m i q^{n-1}. \]
We now claim that
\[ \overline{q^{m-1} i}_m = \overline{j(m)}_m. \]
To prove this we see that $l_m(q^{m-1} i) = 1$ and so $\overline{q^{m-1} i}_m = q^{m-1} i$. Now $l_m(j(m)) = 1$ and
\[ \overline{j(m)}_m = i q^{m-1} \]
which proves the claim. This gives us
\[ \phi\left(\frac{i}{q}\right) = \phi\left( \frac{j(m)}{q^m} \right) \]
for all $m$ and letting $m \to \infty$ we get that $\phi$ is constant on the interval
\[ \left[ \frac{i}{q}, \frac{i}{q-1} \right]. \]
Now letting $m=1$ we find
\[ g_i^1(x) = x - \bar i_1 = x - i \]
which has one root, $x=i$, so we get
\[ \phi\left(\frac{i}{q}\right) = \frac{\log i}{\log q} \]
on this interval.

A bit more work allows us to calculate $\phi(x)$ for $x = \frac{i}{q^n}$ for larger $n$ since we here need to solve polynomial equations of degree $n$.

\section{Numerical plot}
Calculating the spectral radii of $A(k)$, we can make numerical plots
of the function~$\phi$. The plot in figure \ref{fig:2357} was made
using GNU Octave.
\begin{figure}[!htp]
\centering
\includegraphics[scale=1.0]{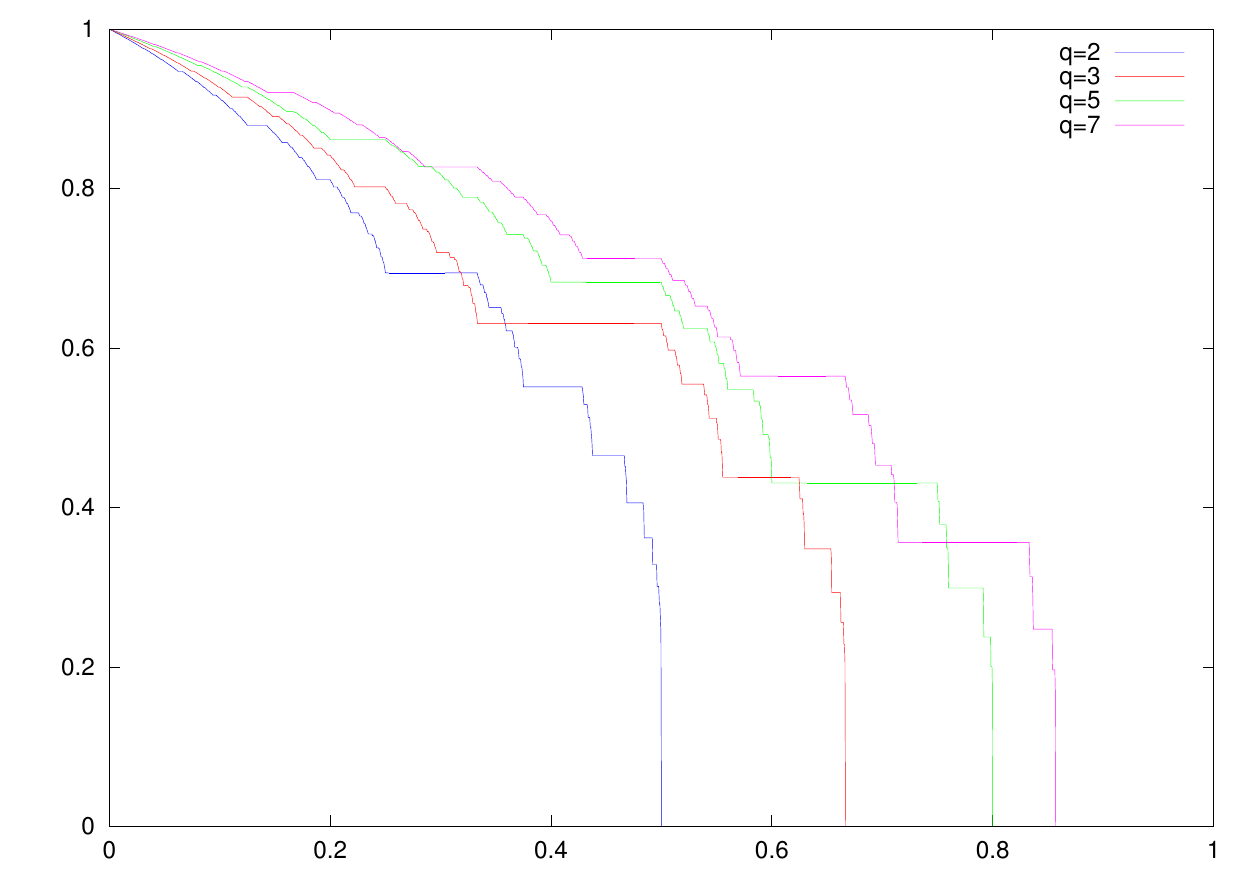}
\caption{Numerical plots of $\phi$ for $q \in \{2,3,5,7\}$.}
\label{fig:2357}
\end{figure}

\section{Asymptotics}
We now want to consider $\phi$ as $q \to \infty$. We consider the function $\psi: [0,1) \to [0,1)$ where
\[ \psi(c) = \left\{ \begin{array}{ll} 
1+\frac{\log(1-c)}{\log q} & 0 \leq c < \frac{q-1}{q} \\
0 & \text{otherwise.}
\end{array}\right. \]
and wish to prove that $\phi$ and $\psi$ are somewhat asymptotically similar. This can also be expressed by saying that $\rho(A_c)$ behaves somewhat like $q-qc$, which is true in the starting point of the intervals where $\phi$ is constant, so we get the following theorem.

\begin{figure}[!htp]
\centering
\includegraphics[scale=1.0]{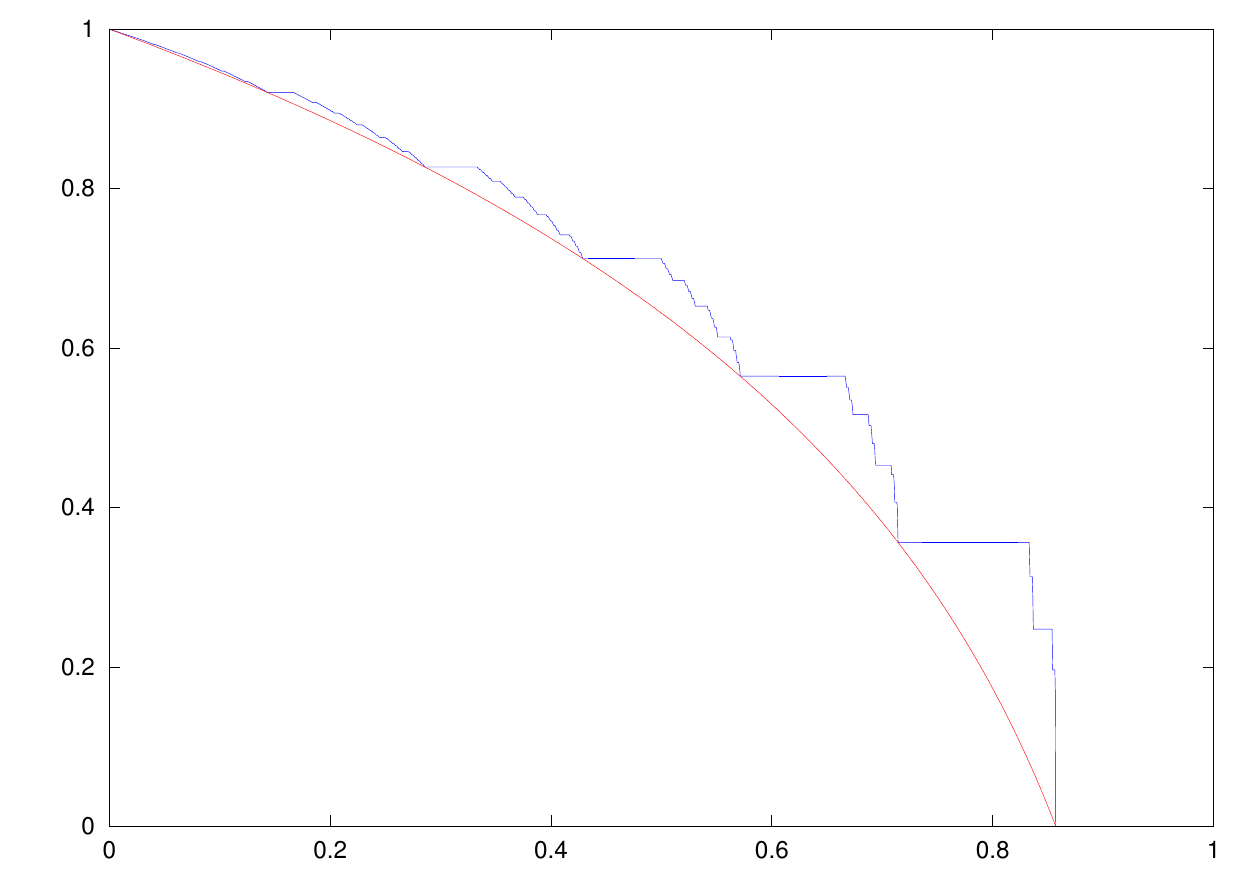}
\caption{Plots of $\phi$ and $\psi$ when $q=7$.}
\end{figure}

\begin{thm} For all $c \in [0,1)$ we have
\[ \frac{\phi(c)}{\psi(c)} \to 1 \]
as $q \to \infty$.
\end{thm}
\begin{proof} Let $c \in [0,1)$ be given. Then if we let $i = \lfloor qc \rfloor$ we have
\[ \frac{i}{q} \leq c \leq \frac{i+1}{q}. \] 
Now
\[ \phi\left(\frac{i}{q}\right) \geq \phi(c) \geq \phi\left(\frac{i+1}{q}\right) \]
and likewise for $\psi$ since both functions are decreasing. Due to the result we got earlier on constant intervals we have
\[ \frac{\log(q-i)}{\log q} \geq \phi(c), \psi(c) \geq \frac{\log(q+1-i)}{\log q} \]
so recalling the definition of $i$ we have
\[ \frac{\log(q-i)}{\log(q-i+1)} \geq \frac{\phi(c)}{\psi(c)} \geq \frac{\log(q-i+1)}{\log(q-i)} \]
and since $i \to \infty$ as $q \to \infty$, both the lower and upper bound converges to $1$. This finishes the proof.
\end{proof}

Since we also see that $\psi(c) \to 1$ as $q \to \infty$, we also have the following corollary.
\begin{cor} For all $c \in [0,1)$ we have
\[ \phi(c) \to 1 \text{ as } q \to \infty. \]
\end{cor}
The convergence is very slow though -- since $\phi$ and $\psi$ are equal on $q$ points we can just look at the convergence of 
\[ \frac{\log(1-c)}{\log q} \]
to zero which is easy to calculate.

\begin{figure}[!htp]
\centering
\includegraphics{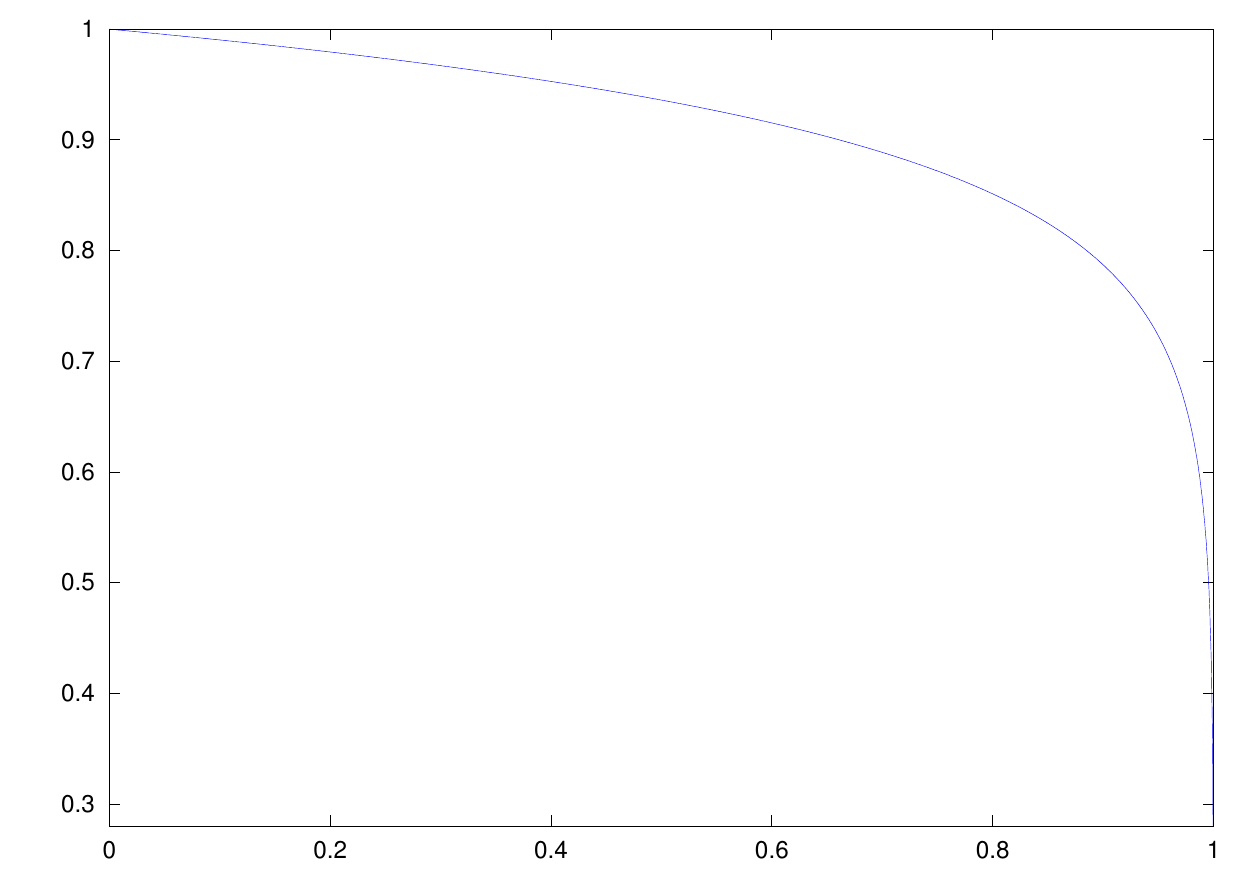}
\caption{Plot of $\phi$ when $q=50000$.}
\end{figure}

\FloatBarrier

\end{document}